\documentclass[proceedings]{stacs}
\stacsheading{2009}{469--480}{Freiburg}
\firstpageno{469}

\usepackage{enumerate}
\usepackage{epsf, verbatim}
\usepackage[latin1]{inputenc}
\usepackage{amssymb}
\usepackage[a4paper,hyperindex,breaklinks,colorlinks,citecolor=blue]{hyperref}

\newcommand{\R}{\mathbb{R}}
\newcommand{\N}{\mathbb{N}}
\newcommand{\Q}{\mathbb{Q}}

\newcommand{\X}{\mathcal{X}}

\renewcommand{\S}{\mathcal{S}}

\newcommand{\FI}{\varphi}

\newcommand{\vers}{\rightarrow}

\newcommand{\comp}[1]{{#1}^\mathcal{C}}
\newcommand{\dom}{\mbox{dom}}

\newcommand{\norm}[1]{\left\Vert #1 \right\Vert}
\newcommand{\defin}[1]{\emph{\textbf{#1}}}
\newcommand{\puce}{$\bullet$ }

\newcommand{\Y}{\mathcal{Y}}

\newcommand{\bbB}{\mathbb{B}}
\newcommand{\bbN}{\mathbb{N}}

\newcommand{\integral}[2]{\int\!{#1}\,\mathrm{d}{#2}}

 \newcommand{\fB}{\mathfrak{B}}

\newcommand{\BC}{\mathrm{BC}}
\renewcommand{\dom}{\mathrm{dom}}
\newcommand{\id}{\mathrm{id}}
\newcommand{\Int}{\mathrm{Int}}

\newcommand{\range}{\mathrm{range}}
\newcommand{\Sch}{\mathrm{Sch}}

\newcommand{\clint}[2]{[#1,#2]}

\newcommand{\len}[1]{\mathopen|#1\mathclose|}

\theoremstyle{definition}

\begin{document}
\title[Randomness on Computable Probability Spaces]{Randomness on Computable Probability Spaces---a dynamical point of view\footnote{partly supported by ANR Grant  05 2452 260 ox}}
\author[1]{P. G\'acs}{Peter G\'acs}
\address[1]{Computer Science Department, Boston University}
\author[2]{M. Hoyrup}{Mathieu Hoyrup}
\address[2]{D\'epartement d'Informatique, \'Ecole Normale Sup\'erieure de Paris}
\author[3]{C. Rojas}{Crist\'obal Rojas}
\address[3]{DI \'Ecole Normale Sup\'erieure and CREA \'Ecole Polytechnique, Paris.}


\begin{abstract} 
We extend the notion of randomness (in the version introduced by Schnorr)
to computable Probability Spaces and
compare it to a \emph{dynamical} notion of randomness: typicality. 
Roughly, a point is \emph{typical} for some dynamic, if it follows the statistical
behavior of the system (Birkhoff's pointwise ergodic theorem). 
We prove that
a point is Schnorr random if and only if it is typical for every \emph{mixing}
computable dynamics.
To prove the result we develop some tools for the theory of computable
probability spaces (for example, morphisms)
that are expected to have other applications.
\end{abstract}

\maketitle


\section{Introduction}
The roots of algorithmic randomness go back to the work of von Mises in the
20th century. 
He suggested a notion of individual infinite random sequence
based on \emph{limit-frequency} properties invariant under the action of
\emph{selection} functions from some ``acceptable'' set. 
The problem was
then to properly define what an ``acceptable'' selection function could be. 
Some years later, the concept of \emph{computable} function was formalized,
providing a natural class of functions to be considered as acceptable. 
This gave rise to Church's notion of \emph{computable randomness}. 
Nevertheless,
substantial understanding was achieved only with the works of Kolmogorov
\cite{Kol65}, Martin-L\"of \cite{MLof66}, Levin \cite{LevZvo70} and Schnorr
\cite{Sch71} and since then, many efforts have contributed to the
development of this theory which is now well established and intensively
studied.

There are several different possible definitions, but it is Martin-L\"of's
one which has received most attention. This notion can be defined, at
least, from three different points of view:

\begin{enumerate}
\item \emph{measure theoretic}. This was the original presentation by
Martin-L\"of (\cite{MLof66}). Roughly, an infinite sequence is random if it
satisfies all ``effective'' probabilistic laws (see definition~\ref{MLrandom}).

\item \emph{compressibility}. This characterization of random sequences,
  due to Schnorr and Levin (see \cite{LevZvo70, Sch72}), uses the prefix-free Kolmogorov
complexity: random sequences are those which are maximally complex.

\item \emph{predictability}. 
In this approach (started by Ville~\cite{Ville39}
and reintroduced to the modern theory by Schnorr~\cite{Sch72}) 
a sequence is random if, in a fair betting game, no ``effective'' strategy
(``martingale'') can win an unbounded amount of money against it.

\end{enumerate}

In~\cite{Sch71}, a somewhat broader notion of algorithmic randomness 
(narrower notion of probabilistic law) was
proposed: Schnorr randomness. This notion received less attention over the years:
Martin-L\"ofs definition is simpler, leads to universal tests, and
many equivalent characterizations (besides, Schnorr's book is not in
English$\dots$). Recently, Schnorr randomness has begun to receive more attention. 
The work~\cite{Dow02} for instance, characterizes it
in terms of Kolmogorov complexity.

In the present paper, first we extend Schnorr randomness
to arbitrary computable probability spaces and develop some useful tools. 
Then, taking a \emph{dynamical systems} point of view, we introduce yet another approach to the
definition of randomness: typicality. 
Roughly, a point is \emph{typical}
for some measure-preserving 
ergodic dynamic, if it follows the statistical behavior of the
system (given by Birkhoff's pointwise ergodic theorem)
with respect to every bounded continous function used to follow its trajectory (or equivalently, every computable function, see Definition~\ref{mutyp}). We then show that:

\bigskip
\noindent \textbf{Theorem.} \emph{In any computable probability space, a point is Schnorr random if and only if
it is typical for every \emph{mixing} computable dynamical system.}
\bigskip

The paper is organized as follows: Section~\ref{preliminaries}
presents all needed concepts of computability theory and computable
measure theory over general metric spaces. Parts of this section, for example on $\mu$-computable functions, are new
and should be of independent interest.
Section~\ref{ran_typ} generalizes Schnorr randomness and studies
some useful properties, after which we introduce the notion of
typicality. 
Section~\ref{mainproof} is devoted to the proof of our main result.

\section{Computability}\label{preliminaries}
In classical recursion theory, a set of natural
numbers is called \defin{recursively enumerable} (\defin{r.e.}~for
short) if it is the range of some partial recursive function. 
That is if there exists an algorithm listing (or enumerating) the set.  

Strictly speaking, recursive functions only work on natural numbers, but
this can be extended to the objects (thought of as ``finite'' objects) of any
countable set, once a numbering of its elements has been chosen. 
We will sometimes use the word \defin{algorithm} instead of \emph{recursive function} when the
inputs or outputs are interpreted as finite objects. 
The operative power of an
algorithm on the objects of such a numbered set obviously depends on what
can be effectively recovered from their numbers.

\vspace*{2.5mm}
\begin{examples}~\\
\hspace*{0.8mm}
\begin{minipage}{0.988\textwidth}
\item[1] $\N^k$ can be numbered in such a way that the $k$-tuple of number $i$ can be computed from $i$ and vice versa. 

\item[2]  The set $\mathbb{Q}$ of rational numbers can be injectively
numbered $\mathbb{Q}=\{q_0,q_1,\ldots\}$ in an \emph{effective} way: the
number $i$ of a rational $a/b$ can be computed from $a$ and $b$, and vice
versa. We fix such a numbering.
\end{minipage}
\end{examples}

All through this work, we will use recursive functions over numbered sets to define
\emph{computability} or \emph{constructivity} notions on \emph{infinite} objects.
Depending on the context, these notions will take particulars names (computable, recursively enumerable, r.e. open,
decidable, etc...) but the definition will be always of the form: \emph{obect $x$ is \defin{constructive} if there exists a recursive $\FI$: $\N \to D$
satisfying \emph{property} P($\FI,x$)} (where $D$ is some numbered set).

For example, \emph{$E\subset \N$ is \defin{r.e.} if there
exists a recursive $\FI$: $\N\to \N$ satisfying $E=\range(\FI)$}.

Each time, a \emph{uniform version} will be implicitly defined: \emph{a sequence $(x_i)_i$ is constructive \defin{uniformly in
$\boldsymbol{i}$} if there exists a recursive $\FI$: $\N\times \N\to D$ satisfying \emph{property} P($\FI(i,\cdot),x_i$) for all $i$}.

In our example, 
\emph{a sequence $(E_i)_i$ is \defin{r.e. uniformly in $\boldsymbol{i}$} if there exists
$\FI$: $\N\times \N\to \N$ satisfying $E_i=\range(\FI(i,\cdot))$ for all $i$}.

Let us ilustrate this in the case of reals numbers (computable reals numbers were introduced by Turing in \cite{Tur36}).

\begin{definition}\index{computable!reals}\label{d.computable_reals}
A real number $x\in \R$ is said to be \defin{computable} if there exists a total
recursive $\FI:\N \to \Q$ satisfying $|x-\FI(n)|<2^{-n}$ for all $n\in \N$.
\end{definition}

Hence by a sequence of reals $(x_i)_i$ computable \defin{uniformly in
$\boldsymbol{i}$} we mean that there exists a recursive $\FI:$ $\N \times \N \to \Q$ satisfying $|x-\FI(i,n)|<2^{-n}$ for
all $n\in \N$, for all $i\in \N$.

We also have the following notions: 

\begin{definition}
Let $x$ be a real number. We say that:

\puce $x$ is \defin{lower semi-computable} if the set $\{i\in \N :q_i<x\}$ is r.e.,

\puce  $x$ is \defin{upper semi-computable} if the set $\{i\in \N :q_i>x\}$ is r.e.,

\end{definition}

It is easy to see that a real number is computable if and only if it is lower and upper semi-computable.


\subsection{Computable metric spaces}
We breifly recall the basic of computable metric spaces. 

\begin{definition}
A \defin{computable metric space} (CMS) is a triple $\X=(X,d,\S)$,  where

\puce $(X,d)$ is a separable complete metric space.

\puce $\S=(s_i)_{i \in \N}$ is a numbered dense subset of $X$ (called  \defin{ideal points}).

\puce The real numbers $(d(s_i,s_j))_{i,j}$ are all computable, uniformly in $i,j$.

\end{definition}

Some important examples of computable metric spaces:

\vspace*{2.5mm}
\begin{examples}~\\
\hspace*{0.8mm}
\begin{minipage}{0.988\textwidth}
\item[1] The Cantor space $\mathcal(\Sigma^{\N},d,S)$ with $\Sigma$ a
finite alphabet.
If $x=x_{1}x_{2}\dots$, $y=y_{1}y_{2}\dots$, are elements then the distance is
defined by $d(x,y) = \sum_{i:x_{i}\ne y_{i}} 2^{-i}$.
Let us fix some element of $\Sigma$ denoting it by $0$.
The dense set $S$ is the set of ultimately $0$-stationary sequences.

\item[2] $(\R^n,d_{\R^n},\Q^n)$ with the Euclidean metric and the standard numbering of
$\Q^n$.
\end{minipage}
\end{examples}

For further examples we refer to \cite{Wei93}.

The numbered set of ideal points $(s_i)_i$ induces the numbered set of
\defin{ideal balls} 
$\mathcal{B}:=\{B(s_i,q_j):s_i \in S, q_j \in\Q_{>0}\}$. We denote by $B_{\langle i,j \rangle}$ (or just $B_{n}$)
the ideal ball $B(s_i,q_j)$, where $\langle \cdot, \cdot \rangle$ is a computable bijection between tuples and integers.

\begin{definition}[Computable points]
A point $x\in X$ is said to be \defin{computable} if the set $E_x:=\{i\in \N : x\in B_i\}$ is r.e.
\end{definition}

\begin{definition}[R.e.~open sets]
We say that the set $U\subset X$ is \defin{r.e.~open}
if there is some r.e.~set $E\subset \N$ such that $U=\bigcup_{i\in E}B_i$. If $U$ is r.e. open and $D\subset X$ is an arbitrary set then the set $A:=U\cap D$ is called \defin{r.e. open in $\boldsymbol{D}$}.
\end{definition}

\vspace*{1mm}
\begin{examples}~\\
\hspace*{0.8mm}
\begin{minipage}{0.988\textwidth}
\item[1] If the sequence 
$(U_n)_n$ is  r.e. \hspace*{-0.5mm}open uniformly in $n$, then the union $\bigcup_n\hspace*{-1mm} U_n$ is an r.e. \hspace*{-0.5mm}open set.

\item[2] $U_i\cup U_j$ and $U_i\cap U_j$ are r.e.~open uniformly in $(i,j)$. 
See~\cite{HoyRoj07}.
\end{minipage}
\end{examples}

Let $(X,S_X,d_X)$ and $(Y,S_Y,d_Y)$ be computable metric spaces. Let $(B^Y_i)_i$ be the collection of ideal balls from $Y$.

\begin{definition}[Computable Functions]\label{functions}
A function $T: X \vers Y$ is said to be \defin{computable} if $T^{-1}(B^Y_i)$ 
is r.e.~open uniformly in $i$.
\end{definition}

It follows that computable functions are continuous. 
Since we will
work with functions which are not necessarily continuous everywhere (and hence not computable),
we shall consider functions which are computable on some subset of $X$. 
More precisely, a function $T$ is said to be \defin{computable
on D} ($D\subset X$) if $T^{-1}(B^Y_i)$ is r.e. open in $D$, uniformly in $i$.
The set $D$ is called the \defin{domain of computability} of $T$.


\section{Computable Probability Spaces\label{seccompmu}}

Let us recall some basic concepts of measure theory. 
Let $X$ be a set. A family $\fB$ of subsets of $X$ is called an \defin{algebra} if (i)$X \in  \fB$, (ii)$ A\in \fB \Rightarrow \comp{A} \in \fB$ and (iii) $A,B\in \fB   \Rightarrow A\cup B \in \fB$. We say that $\fB$ is a $\boldsymbol{\sigma}$-\defin{algebra} if moreover $A_i\in \fB, i \geq 1 \Rightarrow \bigcup_{i} A_i \in \fB$. If $\fB_0$ is a family of subsets of $X$, the $\sigma$-algebra generated by
$\fB_0$ (denoted $\sigma(\fB_0)$) is defined to be the smallest $\sigma$-algebra
over $X$ that contains $\fB_0$. If $\fB$ is a $\sigma$-algebra of subsets of $X$, we say that 
$\mu: \fB\to[0,1]$ is a \defin{probability measure} if, for every family 
$(A_i)_i\subset \fB$ of disjoint subsets of $X$, the following holds: 
\begin{align}\label{e.sum-cond}
  \mu \mbox{\LARGE{(}} \bigcup_i A_i \mbox{\LARGE{)}} = \sum_i \mu(A_i).
\end{align}

If $X$ is a topological space, the \defin{Borel} $\sigma$-algebra of $X$ is
defined as the $\sigma$-algebra generated by the family of open sets of
$X$. Sets in the Borel $\sigma$-algebra are called Borel sets. 
In this paper, a
\defin{probability space} will always refer to the triple $(X,\fB,\mu)$, where
$\fB$ is the Borel $\sigma$-algebra of $X$ and $\mu$ is a probability measure. 
A set $A\subset X$ has \defin{measure zero} if there is a Borel set $A_1$ such
that $A\subset A_1$ and $\mu(A_1)=0$. 
We call two sets $A_1,A_2 \subset X$
\defin{equivalent modulo zero}, and write $A_1=A_2 \pmod{0}$, if the symmetric
difference has measure zero. 
We write $A_1\subset A_2\pmod{0}$ if $A_1$ is a subset of $A_2$ and 
$A_1=A_2 \pmod{0}$. 

When $X$ is a computable metric space, the space of probability measures
over $X$, denoted by $\mathcal{M}(X)$, can be endowed with a structure of
computable metric space. 
Then a computable measure can be defined as a
computable point in $\mathcal{M}(X)$.

\begin{example}[Measure over a Cantor space]\label{x.Cantor-measure}

As a special example, we can set $X=\bbB^\N$ where $\bbB=\{0,1\}$ and 
$\lambda([x]) = 2^{-\len{x}}$, where $\len{x}$ is the length of the binary string $x\in \{0,1\}^*$.
This is the distribution on the set of infinite binary sequences obtained by
tossing a fair coin, and condition~\eqref{e.sum-cond} simplifies to
 \begin{align*}
   \lambda(x0)+\lambda(x1) = \lambda(x).
 \end{align*}
\end{example}

Let $\mathcal{X}=(X,d,S)$ be a computable metric space. Let us consider the
space $\mathcal{M}(X)$ of measures over $X$ endowed with weak topology, that
is: 
\begin{equation*}
\mu_n\to\mu \mbox{ iff } \mu_n f \to \mu f 
\mbox{ for all real continuous bounded } f,
\end{equation*}
where $\mu f$ stands for $\integral{f}{\mu}$. 

If $X$ is separable and complete, then $\mathcal{M}(X)$ is separable and
complete. 
Let $D\subset \mathcal{M}(X)$ be the set of those probability
measures that are concentrated in finitely many points of $S$ and assign
rational values to them. 
It can be shown that this is a dense subset (\cite{Bil68}).

We consider the Prokhorov metric $\rho$ on $\mathcal{M}(X)$ defined by: 
\begin{equation*}
\rho(\mu,\nu):=\inf \{\epsilon \in \mathbb{R}^+ :
\mu(A)\leq\nu(A^{\epsilon})+\epsilon \mbox{ for every Borel set }A\}
\end{equation*}
where $A^{\epsilon}=\{x:d(x,A)< \epsilon\}$.

This metric induces the weak topology on $\mathcal{M}(X)$. 
Furthermore, it
can be shown that the triple $(\mathcal{M}(X),D,\rho)$ is a computable
metric space (see~\cite{Gac05}, \cite{HoyRoj07}).

\begin{definition}\label{compmeas}
A measure $\mu $ is computable if it is a computable point of $(\mathcal{M}(X),D,\rho)$
\end{definition}

The following result (see \cite{HoyRoj07}) will be intensively used in the sequel:

\begin{lemma}
\label{lower_semi_compute} A probability measure $\mu$ is computable if and
only if the measure of finite union of ideal balls 
$\mu(B_{i_1}\cup\ldots \cup B_{i_k})$ is lower semi-computable, uniformly in
$i_1,\ldots,i_k$. 
\end{lemma}

\begin{definition}
A \defin{computable probability space (CPS)} is a pair $(\X,\mu)$
where $\X$ is a computable metric space and $\mu$ is a computable
Borel probability measure on $X$.
\end{definition}

As already said, a computable function defined on the whole space is necessarily
continuous. But a transformation or an observable need not be continuous at every point,
as many interesting examples prove (piecewise-defined transformations,
characteristic functions of measurable sets,\dots), so the requirement of
being computable everywhere is too strong. 
In a measure-theoretical setting,
the natural weaker condition is to require the function to be computable 
\emph{almost everywhere}. 
In the computable setting this is not enough, and a computable condition
on the set on which the function is computable is needed:

\begin{definition}[Constructive $G_{\delta}$-sets] 
We say that the set $D\subset X$ is a 
\defin{constructive $\boldsymbol{G_\delta}$-set} if it is the intersection of a
sequence of uniformly r.e.~open sets. 
\end{definition}

\begin{definition}[$\mu$-computable functions]
Let $(\X,\mu)$ and $\Y$ be a CPS and a CMS respectively.
A function $f:(\X,\mu)\rightarrow Y$ is \defin{$\mu$-computable} if it is 
computable on a constructive $G_{\delta}$-set (denoted as $\dom f$ or
$D_f$) of measure one.
\end{definition}

\begin{example}\label{binary_expansion}
Let $m$ be the Lebesgue measure on $[0,1]$. The binary expansion of reals defines a function from non-dyadic numbers to infinite binary sequences which induces a $m$-computable function from  $([0,1],m)$ to $\{0,1\}^\N$.
\end{example}

\begin{remark}\label{com_sum}
Given a uniform sequence of $\mu$-computable functions $(f_i)_i$, any computable 
operation $\odot_{i=0}^n f_i$ (adition, multiplication, composition, etc...) is
$\mu$-computable, uniformly in $n$.  
\end{remark}

We recall that $F: (\X,\mu)\to (\Y,\nu)$ is measure-preserving if 
$\mu(F^{-1}(A))=\nu(A)$ for all Borel sets $A$.

\begin{definition}[morphisms of CPS's]
A \defin{morphism of CPS's} 
$F:(\X,\mu)\to(\mathcal{Y},\nu)$, is a
$\mu$-computable measure-preserving function $F:D_F\subseteq X\to Y$.

An \defin{isomorphism of CPS's} 
$(F,G):(\X,\mu)\rightleftarrows(\mathcal{Y},\nu)$ is a pair 
$(F,G)$ of morphisms such that $G\circ F=\id$ on $F^{-1}(D_G)$ and 
$F\circ G=\id$ on $G^{-1}(D_F)$.
\end{definition}

\begin{example}\label{x.Cantor-isom}
Let $(\bbB^{\bbN},\lambda)$ the probability space introduced in
Example~\ref{x.Cantor-measure} with the coin-tossing distribution $\lambda$ over the
infinite sequences.
The binary expansion (see example \ref{binary_expansion}) creates an isomorphism of CPS's between the spaces $(\clint{0}{1},m)$ and $(\bbB^{\bbN},\lambda)$.
\end{example}

\begin{remark}
To every isomorphism of CPS's $(F,G)$ one can associate the canonical invertible morphism of CPS's 
$\FI=F|_{D_{\FI}}$ with $\FI^{-1}=G|_{D_{\FI^{-1}}}$, where 
$D_{\FI}=F^{-1}(G^{-1}(D_F))$ and $D_{\FI^{-1}}=G^{-1}(D_F)$. 
Of course, $(\FI,\FI^{-1})$ is an isomorphism of CPS's as well.
\end{remark}

The next proposition is a direct consequence of theorem 5.1.1 from \cite{HoyRoj07}:

\begin{proposition}\label{isomorphism}
Every computable probability space is isomorphic to the Cantor space with an
appropiate computable measure. 
\end{proposition}

\begin{definition}
A set $A\subset X$ is said to be \defin{almost decidable} if the function $1_A:$ $X\to \{0,1\}$ is $\mu$-computable. 
\end{definition}

It is easy to see that a set $A$ is almost decidable iff there is a constructive
$G_{\delta}$ set $D$ of measure one and two r.e.~open sets $U$ and $V$ such that: 
$$
U\cap D \subset A, \quad V\cap D \subseteq \comp{A}, \quad \mu(U)+\mu(V)=1.
$$

\begin{remarks}
\item[1] The collection of almost decidable sets is an algebra.
\item[2] An almost decidable set is always a continuity set.
\item[3] Ideal balls with zero boundary measure are always almost decidable.
\item[4] Unless the space is disconnected (i.e.~has
non-trivial clopen subsets), no set can be \emph{decidable}, 
i.e.~semi-decidable (r.e.) and with a semi-decidable complement (such a set must
be clopen\footnote{In the Cantor space for example (which is totally
disconnected), every cylinder (ball) is a decidable set. 
Indeed, to decide
if some infinite sequence belongs to some cylinder it suffices to
compare the finite word defining the cylinder to the corresponding
finite prefix of the infinite sequence.}). 
Instead, a set can be decidable \emph{with probability $1$}: 
there is an algorithm which decides if a point belongs to the set or not, 
for almost every point. 
This is why we call it \emph{almost decidable}.
\end{remarks}

Ignoring computability,
the existence of open sets with zero boundary measure directly follows from the fact that
the collection of open sets is uncountable and $\mu$ is finite. 
The problem in the computable setting is that there are only countable many open 
r.e.~sets. 
Fortunately, there still always exists a basis of almost decidables balls. 

\begin{lemma}\label{zero_measure_points}
Let $X$ be $\R$ or $\R^+$ or $[0,1]$. 
Let $\mu$ be a computable probability measure on $X$. 
Then there is a sequence of uniformly computable reals $(x_n)_n$
which is dense in $X$ and such that $\mu(\{x_n\})=0$ for all $n$.
\end{lemma}
\begin{proof}
Let $I$ be a closed rational interval. 
We construct $x\in I$ such that $\mu(\{x\})=0$. 
To do this, we construct inductively a nested sequence of
closed intervals $J_k$ of measure $<2^{-k+1}$, with $J_0=I$. 
Suppose $J_k=[a,b]$ has been constructed, with $\mu(J_k)<2^{-k+1}$. 
Let
$m=(b-a)/3$: one of the intervals $[a,a+m]$ and $[b-m,b]$ must have measure
$<2^{-k}$, and since their measure is upper-computable, we can find it effectively---let it be $J_{k+1}$.

From a constructive enumeration $(I_n)_n$ of all the dyadic intervals, we
can construct $x_n\in I_n$ uniformly.
 \end{proof}

\begin{corollary}\label{sequence}
Let $(\X,\mu)$ be a CPS and $(f_i)_i$ be a sequence of uniformly computable real
valued functions on $X$. 
Then there is a sequence of uniformly computable reals $(x_n)_n$ which is
dense in $\R$ and such that $\mu(\{f_i^{-1}(x_n)\})=0$ for all $i,n$.
\end{corollary}

\begin{proof}
Consider the uniformly computable measures $\mu_i=\mu\circ f_i^{-1}$ and define
$\nu=\sum_i2^{-i}\mu_i$. 
By Lemma \ref{lower_semi_compute}, $\nu$ is a computable measure and then, by
Lemma~\ref{zero_measure_points}, there is a sequence of uniformly computable
reals $(x_n)_n$ which is dense in $\R$ and such that $\nu(\{x_n\})=0$ for all
$n$. 
Since $\nu(A)=0$ iff $\mu_i(A)=0$ for all $i$, we get $\mu(\{f_i^{-1}(x_n)\})=0$
for all $i,n$.
\end{proof}

The following result will be used many times in the sequel.

\begin{corollary}\label{a_d_balls}
There is a sequence of uniformly computable reals $(r_n)_{n\in \N}$ such
that $(B(s_i,r_n))_{i,n}$ is a basis of almost decidable balls.
\end{corollary}
\begin{proof}
Apply Corollary~\ref{sequence} to $(f_i)_i$ defined by $f_i(x)=d(s_i,x)$.
\end{proof}

We remark that every ideal ball can be expressed as a r.e.~union of almost 
decidable balls, and vice-versa. 
So the two bases are constructively equivalent.

\begin{definition}
A computable probability space is a \defin{computable Lebesgue space} if it is
isomorphic to the computable probability space $([0,1],m)$ where $m$
is the Lebesgue measure. 
\end{definition}

\begin{theorem}\label{lebesgue}
Every computable probability space with no atoms is a computable Lebesgue space.
\end{theorem}

\begin{proof}
We first prove the result for $I=([0,1],\mu)$.

\begin{lemma}
The interval endowed with a non-atomic computable probability measure is a
computable Lebesgue space.
\end{lemma}

\begin{proof}
We define the morphism of the CPS as $F(x)=\mu([0,x])$. 
As $\mu$ has no atom and
is computable, $F$ is computable and surjective.
As $F$ is surjective, it has right inverses. 
Two of them are $G_<(y)=\sup\{x:F(x)<y\}$ and $G_>(y)=\inf\{x:F(x)>y\}$, 
and satisfy $F^{-1}(\{y\}) = [G_<(y),G_>(y)]$. 
They are increasing and respectively left- and right-continuous. 
As $F$ is computable, they are even lower- and upper semi-computable respectively.
Let us define $D=\{y:G_<(y)=G_>(y)\}$: every $y\in D$ has a unique pre-image by
$F$, which is then injective on $F^{-1}(D)$. 
The restriction of $F$ on $F^{-1}(D)$ has a left-inverse, which is given by the
restriction of $G_<$ and $G_>$ on $D$. 
Let us call it $G:D\to I$. 
By lower and
upper semi-computability of $G_<$ and $G_>$, $G$ is computable. 
Now, $D$ is a constructive $G_{\delta}$-set: 
$D=\bigcap_n\{y:G_>(y)-G_<(y)<1/n\}$. 
We show that $I\setminus D$ is a countable set. 
The family $\{[G_<(y),G_>(y)]:y\in I\}$ indexed by $I$ is a family of disjoint
closed intervals, included in $[0,1]$. 
Hence, only countably many of them have positive length. 
Those intervals correspond to points $y$ belonging to
$I\setminus D$, which is then countable. 
It follows that $D$ has Lebesgue measure one (it is even dense). 
$(F,G)$ is then an isomorphism between $(I,\mu)$ and $(I,m)$.
\end{proof}

Now, we know from Theorem~\ref{isomorphism} that every CPS $(\X,\mu)$ has a binary
representation, which is in particular an isomorphism with the Cantor space
$(\bbB^{\bbN},\mu')$. As mentioned in Example~\ref{x.Cantor-isom}, the
latter is isomorphic to $(I,\mu_I)$ where $\mu_I$ is the induced measure.
If $\mu$ is non-atomic, so is $\mu_I$. 
By the previous lemma, $(I,\mu_I)$ is isomorphic to $(I, m)$.
\end{proof}

\subsection{Randomness and typicality}\label{ran_typ}

\subsubsection{Algorithmic randomness}

\begin{definition}\label{MLrandom}
 A \defin{Martin-L\"of test} ($ML$-test) is an uniform sequence $(A_n)_n$ 
of r.e.~open sets such that $\mu(A_n)\leq 2^{-n}$. 
We say that $x$ \defin{fails} the $ML$-test if $x\in A_n$ for all $n$. 
A point $x$ is called \defin{ML-random} if it fails no $ML$-test.
\end{definition}

\begin{definition}
 A \defin{Borel-cantelli test} ($\BC$-test) is a uniform sequence $(C_n)_n$ of
r.e.~open sets such that $\sum_n\mu(C_n)<\infty$. 
We say that $x$ \defin{fails} the $\BC$-test if $x\in C_n$ infinitly often (i.o.).
\end{definition}

It is easy to show that:
\begin{proposition}
 $x$ fails a $ML$-test iff $x$ fails a $\BC$-test.
\end{proposition}

\begin{definition}
 A \defin{Schnorr test} ($\Sch$-test) is a ML-test $(A_n)_n$ such that the
sequence of  reals $(\mu(A_n))_n$ is uniformly computable. 
We say that $x$ fails the $\Sch$-test if $x\in A_n$ for all $n$. 
A point $x$ is called \defin{$\Sch$-random} if it fails no $\Sch$-test.
\end{definition}

\begin{definition}
 A \defin{strong $\BC$-test} is a $\BC$-test $(C_n)_n$ such that $\sum_n\mu(C_n)$ is
 computable.
\end{definition}

\begin{proposition}
An element $x$ fails a $\Sch$-test if and only if $x$ fails a strong $\BC$-test.
\end{proposition}
\begin{proof}
Let $(C_n)_n$ be a strong $\BC$-test. 
Let $c$ be such that $2^c>\sum_n\mu(C_n)$. 
Define the r.e.~open set $A_k:=\{x:\vert\{n:x\in C_n\}\vert\geq 2^{k+c}\}$. 
Then $\mu(A_k)<2^{-k}$. 
Observe that $A_k$ is the union of all the ($2^{k+c}$)-intersections of
$C_n$'s. 
Since $\mu(C_k)=\sum_n\mu(C_n)-\sum_{n\neq k}\mu(C_n)$ and the $C_n$'s
are r.e.~we have that $\mu(C_n)$ is computable (uniformly in $n$). 
We choose a basis $(B^i)_i$ of almost decidable balls to work with. 
Recall that finite
unions or intersections of almost decidable sets are almost decidable too and
that the measure of an almost decidable set is computable. 
Now we show that $\mu(A_k)$ is computable uniformly in $k$. 
Let $\epsilon>0$ be rational. 
Let $n_0$ be such that $\sum_{n\geq n_0}\mu(C_n)<\frac{\epsilon}{2}$. 
Then $\mu(\bigcup_{n\geq n_0}C_n)<\frac{\epsilon}{2}$. 
For each $C_n$ with $n<n_0$ we 
construct an almost decidable set $C_n^{\epsilon}\subset C_n$ (a finite union
of almost decidable balls) such that
$\mu(C_n)-\mu(C_n^{\epsilon})<\frac{1}{n_0}\frac{\epsilon}{2}$. 
Then $\sum_{n<n_0}[\mu(C_n)-\mu(C_n^{\epsilon}]<\frac{\epsilon}{2}$. 
Define $A_k^{\epsilon}$ to be the union of the ($2^{k+c}$)-intersections of the
$C_n^{\epsilon}$'s for $n<n_0$. 
Then $A_k^{\epsilon}$ is almost decidable and then has a computable measure. 
Moreover $A_k\subset A_k^{\epsilon}\cup
(\bigcup_{n\geq n_0}C_n) \cup (\bigcup_{n<n_0}C_n\setminus C_n^{\epsilon})$, hence
$\mu(A_k)-\mu(A_k^{\epsilon})<\epsilon$. 
\end{proof}

The following result is an easy modification of a result from~\cite{HoyRoj07},
so we omit the proof.

\begin{proposition}
Morphisms of computable probability spaces are defined (and computable) on Schnorr random points
and preserve $\Sch$-randomness. 
\end{proposition} 

\subsection{Dynamical systems and typicality}

Let $X$ be a metric space, let $T:X\mapsto X$ be a Borel map. 
Let $\mu $ be
an invariant Borel measure on $X$, that is: $\mu (A)=\mu (T^{-1}(A))$ holds
for each measurable set $A$. 
A set $A$ is called $T$-invariant if $T^{-1}(A)=A$ modulo a set of measure 0.
The system $(T,\mu )$ is said to be ergodic if each 
$T$-invariant set has total or null measure. 
In such systems the famous
Birkhoff ergodic theorem says that time averages computed along 
$\mu$-typical orbits coincide with space averages with respect to $\mu$. 
More precisely, for any $f\in L^{1}(X)$ it holds 
\begin{equation}
\underset{n\rightarrow \infty }{\lim }\frac{S_{n}^{f}(x)}{n}= \integral{f}{\mu},
\label{Birkhoff}
\end{equation}%
for $\mu$-almost each $x$, where 
$S_{n}^{f}=f+f\circ T+\ldots +f\circ T^{n-1}$.

If a point $x$ satisfies equation~(\ref{Birkhoff}) for a certain $f$, then we
say that $x$ is \defin{typical} with respect to the \defin{observable} $f$.

\begin{definition}\label{mutyp}
If $x$ is typical w.r.~to every bounded continuous function $f:X\to \mathbb{R}$, 
then we call it a $T$-\defin{typical point}.
\end{definition}

\begin{remark}
The proof of our main theorem will show as a side result that the definition
would not change if we replaced ``continuous'' with ``computable'' in it.
\end{remark}

In \cite{Vyu97} is proved that ML-random infinite binary sequences are typical w.r.~to any computable
$f$. In \cite{GalHoyRoj1}, this is generalized via effective symbolic dynamics to computable probability spaces and $\mu$-computable observables.

To have the result for $\Sch$-random points it seems that a certain ``mixing'' property or ``loss of memory'' of the system has to be required. This is naturally expressed by means of the \defin{correlation functions}.
For measureable functions $f,g$ let 
\begin{alignat*}{3}
             &C(f,g)       &&=\mu(f\cdot g)-\mu f \cdot\mu g,
\\         &C_{n}(f,g) &&=C(f\circ T^{n},g).
\end{alignat*}
For events $A,B$ with indicator functions $1_{A},1_{B}$ let
\begin{equation*}
 C_{n}(A,B) = C_{n}(1_{A}, 1_{B}),
\end{equation*}
which measures the dependence between the events $A$ and $B$ at times 
$n \gg 1$ and $0$ respectively. 
Note that $C_n(A,B)=0$ corresponds, in probabilistic terms, to 
$T^{-n}(A)$ and $B$ being independent events.

Let us say that a family of Borel sets $\mathcal{E}$ is \defin{essential}, if
for every open set $U$ there is a sequence $(E_i)_i$ of borel sets in $\mathcal{E}$ such
that $\cup_iE_i\subset U\pmod{0}$ (see Section \ref{seccompmu}).

\begin{definition}\label{polynomial decay}
We say that a system $(X,T,\mu )$ is (polynomially) \defin{mixing}
if there is $\alpha >0$ and an essential family $E=\{E_1,E_2,...\}$ 
of almost decidable events such that for each $i,j$ there is $c_{i,j}>0$
computable in $i,j$ such that 
\begin{equation*}
|C_{n}(E_{i},E_{j})|\leq \frac{c_{i,j}}{n^{\alpha }}\qquad \mbox{ for all }n\geq 1.
\end{equation*}
We say that the system is~\defin{independent} if all correlation functions
$C_{n}(E_{i},E_{j})$ are 0 for sufficiently large $n$.
\end{definition}

Examples of non-mixing but still ergodic systems are given for instance by
irrational circle rotations with the Lebesgue measure. 
Examples of mixing but
not independent sytems are given by piecewise expandings maps or uniformly
hyperbolic systems which have a distinguished ergodic measure (called SRB
measure and which is ``physical'' in some sense) with respect to which the
correlations decay exponentially (see~\cite{Via97}). 
An example of a mixing
system for which the decrease of correlations is only polynomial and not
exponential, is given by the class of~\emph{Manneville-Pomeau} type maps (non
uniformly expanding with an indifferent fixed point, see \cite{Iso05}). 
For a survey see \cite{Y}.

\subsection{Proof of the main result}\label{mainproof}

Now we prove our main theorem.

\begin{theorem}\emph{Let $(\X,\mu)$ be a computable probability space with no atoms
The following properties of a point $x\in X$ are equivalent.
 \begin{enumerate}[\upshape (i)]
   \item\label{i.Sch-rand} $x$ is Schnorr random.
   \item\label{i.mixing} $x$ is $T$-typical for every mixing endomorphism $T$.
   \item\label{i.indep} $x$ is $T$-typical for every independent endomorphism $T$.
 \end{enumerate}
}
\end{theorem}

\begin{remark}
If the measure $\mu$ is atomic, it is easy to see that:
\begin{enumerate}
 \item $(\X,\mu)$ admits a mixing endomorphism if and only if $\mu=\delta_x$ for some $x$. In this case the theorem still holds, the only random point being $x$.
 \item $(\X,\mu)$ admits an ergodic endomorphism if and only if $\mu=\frac{1}{n}(\delta_{x_1}+...+\delta_{x_n})$ (where $x_i\neq x_j$, for all $i\neq j$). In this case, a point $x$ is Schnorr random if and only if it is typical for every ergodic endomorphism if and only if it is an atom.\end{enumerate}
\end{remark}

\begin{proof}

Let us first prove a useful lemma. 
Let $E\subset X$ be a Borel set. 
Denote by $1_E$ its indicator function.
The ergodic theorem says that the following equality holds for almost every point:

\begin{equation}\label{birk}
\lim_n \frac{1}{n} \sum_{i=0}^{n-1} 1_E\circ T^i(x) = \mu(E).
\end{equation}

\begin{lemma}\label{extension}
Let $\mathcal{E}$ be an essential family of events. 
If $x$ satisfies equation~(\ref{birk}) for all $E\in \mathcal{E}$ %
then $x$ is a $T$-typical point.
\end{lemma}

\begin{proof}
We have to show that equation~(\ref{birk}) holds for any bounded continuous 
observable $f$.
First, we extend equation~(\ref{birk}) to every continuity open set $C$. 
Let $(E_{i})_i$ be a sequence of elements of $\mathcal{E}$ such that 
$\bigcup_i E_{i} \subseteq \Int(C)$ and $\mu(\bigcup_i E_{i})=\mu(C)$.  
Define $C_k=\bigcup_{i\leq k}E_{i}$. Then $\mu(C_k)\nearrow \mu(C)$. 
For all $k$:
\[
\liminf_n \frac{1}{n} \sum_{i=0}^{n-1} 1_{C}\circ T^i(x)\geq 
\lim_n \frac{1}{n} \sum_{i=0}^{n-1} 1_{C_k}\circ T^i(x)=\mu(C_k)
 \]
so $\liminf_n \frac{1}{n} \sum_{i=0}^{n-1} 1_{C}\circ T^i(x)\geq \mu(C)$. Applying the same argument to $X\setminus C$ gives the result.

Now we extend the result to bounded continuous functions. 
Let $f$ be continuous and bounded ($|f|< M$) and let $\epsilon > 0$ be a
real number.
Then, since the measure $\mu$ is finite, there exist
real numbers $r_1,\dots, r_k \in [-M,M]$ (with $r_1=-M$ and $r_k=M$)
such that $|r_{i+1}-r_i|<\epsilon$ for all $i=1,\dots,k-1$ and
$\mu(f^{-1}(\{r_i\}))=0$ for all $i=1,\dots,k$. 
It follows that for $i=1,\dots,k-1$ the sets $C_i=f^{-1}(]r_i,r_{i+1}[)$ are all
continuity open sets.

Hence the function $f_{\epsilon}=\sum_{i=1}^{k-1}r_i1_{C_i}$ satisfies 
$\norm{f-f_{\epsilon}}_{\infty}\leq \epsilon$ and then the result
follows by density.
\end{proof}

We are now able to prove that $\boldsymbol{\protect\eqref{i.Sch-rand}\Rightarrow
\protect\eqref{i.mixing}}$. 

Let $E \in \mathcal{E}$. 
Put $f = 1_E$.
Observe that $f$ is $\mu$-computable. 
For $\delta >0$, define the deviation sets: 
\begin{equation*}
A_{n}^{f }(\delta )=\left\{ x\in X:\left\vert \frac{S_{n}^{f }(x)}{n}%
- \integral{f}{\mu} \right\vert > \delta \right\} .
\end{equation*}
By Corollary~\ref{sequence} we can choose $\delta$ such that 
$A_{n}^{f}(\delta )$ is almost decidable. 
Then their measures are computable, uniformly in $n$. 

By the Chebychev inequality, $\mu (A_{n}^{f}(\delta ))\leq \frac{1}{\delta^2 } 
 \left\Vert \frac{S_{n}^{f}(x)}{n}-\integral{f}{\mu}\right\Vert _{L^{2}}^2$. Let us change $f$ by adding a constant to have $\integral{f}{\mu} =0$.
This does not change the above quantity.
Then, by invariance of $\mu$ we have
\begin{equation*}
\left\Vert \frac{S_{n}^{f }(x)}{n}-\integral{f}{\mu} \right\Vert_{L^{2}}^{2}
=\integral{ \left( \frac{S_{n}^{f }(x)}{n}\right)^{2}}{\mu} 
=\frac{1}{n^{2}}\integral{nf^{2}}{\mu} +\frac{2}{n^{2}}\integral{\big(\underset{i<j< n}
{\sum }f\circ T^{j-i}f\big)}{\mu}
\end{equation*}
and hence
\begin{eqnarray*}
\delta^2 \mu (A_{n}^{f}(\delta )) \leq 
\frac{\left\Vert f\right\Vert _{L^{2}}^{2}}{n}+\frac{2}{n}\sum_{k<n}|C_k(f,f)| \leq \frac{\left\Vert f\right\Vert _{L^{2}}^{2}}{n}
+\frac{2c_{f,f}}{(1-\alpha)n^\alpha}.
\end{eqnarray*}
(Observe that $\alpha$ can be replaced by any smaller positive number, so we
assume $\alpha<1$.) 
Hence, $\mu (A_{n}^{f}(\delta ))\leq C n^{-\alpha}$
for some constant $C$. 
Now, it is easy to find a sequence $(n_i)_{i\in \mathbb{N}}$ 
such that the subsequence $(n_i^{-\alpha})_i$ is effectively
summable and $\frac{n_i}{n_{i+1}} \rightarrow 1$ (take for instance 
$n_i=i^\beta$ with $\alpha\beta>1$). 
This shows that the sequence $A_{n_i}^{f }(\delta )$ is a strong $\BC$-test. 
Therefore, if $x$ is $\Sch$-random then $x$ belongs to only finitely many 
$A_{n_i}^{f}(\delta )$ for any $\delta$ and hence the subsequence 
$\frac{S_{n_i}^{f}(x)}{n_i}$ converges to $\integral{f}{\mu} =\mu(E)$. 
To show that for such
points the whole sequence $\frac{S_{n}^{f}(x)}{n}$ converges to 
$\integral{f}{\mu}=\mu(E)$, observe that if $n_i\leq n<n_{i+1}$ and 
$\beta_{i}:=\frac{n_i}{n_{i+1}}$ then we have:
\begin{equation*}
\frac{S_{n_i}^{f}}{n_i}-2(1-\beta_{i})M\leq \frac{S_{n}^{f}}{n}\leq 
\frac{S_{n_{i+1}}^{f}}{n_{i+1}}+2(1-\beta_{i})M,
\end{equation*}
where $M$ is a bound of $f$. 
To see this, for any $k,l,\beta$ with $\beta\leq k/l \leq 1$:
\begin{equation*}
\frac{S_{k}^{f}}{k}-\frac{S_{l}^{f}}{l} = \left( 1-\frac{k}{l}\right) 
\frac{S_{k}^{f}}{k}-\frac{S_{l-k}^{f}\circ T^{l-k}}{l} \leq (1-\beta )M+\frac{(l-k)M}{l}=2(1-\beta)M.
\end{equation*}
Taking $\beta =\beta _{i}$ and $k=n_{i}$, $l=n$ first and then $k=n$, 
$l=n_{i+1}$ gives the result. 
Thus, we have proved that a Schnorr random point
$x$ satisfies equation~(\ref{birk}) for any $E \in \mathcal{E}$. 
Lemma~\ref{extension} allows to conclude. 

The \texorpdfstring{$\boldsymbol{\protect\eqref{i.mixing}\Rightarrow
\protect\eqref{i.indep}}$}{(iii) from (ii)} part follows since any independent dynamic is in particular mixing.

To prove the \texorpdfstring{$\boldsymbol{\protect\eqref{i.indep}\Rightarrow
\protect\eqref{i.Sch-rand}}$}{(i) from (iii)} part we will need the following proposition which is a strengthening of a result of Schnorr in~\cite{Sch71}. The proof is somewhat technical, for lack of space we do not included here (see appendix).

\begin{proposition}\label{schacs}
If the infinite binary string $\omega \in (\bbB^{\bbN},\lambda)$ 
is not Schnorr random (w.r.~to the uniform measure),
then there exists an isomorphism 
$\Phi:(\bbB^{\bbN},\lambda)\to (\bbB^{\bbN},\lambda)$ such that
$\Phi(\omega)$ is not typical for the shift transformation $\sigma$. 
\end{proposition}

Now we are able to finish the proof of our main result: suppose that $x$ is not Schnorr random. We construct a dynamic $T$ for which $x$ is not $T$-typical. 
From Proposition~\ref{isomorphism} and Theorem~\ref{lebesgue} we
know that there is an isomorphism 
$\eta:(\X,\mu)\to (\bbB^{\bbN},\lambda)$ (here, $\lambda$ denotes the uniform measure). 
If $x\notin \dom(\eta)$, we can take any independent endomorphism and modify it in order to be the identity on $x$. 
It is cleary still an independent endomorphism (maybe with a smaller domain of computability) and $x$, being a fixed
point, can't be $T$-typical.
 So let $x\in \dom(\eta)$. 
Then $\eta(x)$ is not Schnorr
random in $(\bbB^{\bbN},\lambda)$, since $\eta$ as well as its inverse preserve Schnorr
randomness. 
Then, by Proposition~\ref{schacs}, $\Phi(\eta(x))$ is not
$\sigma$-typical, where $\sigma$ is the shift which is clearly independent
(cylinders being the essential events). 
Put $\psi =\Phi\circ \eta$.
Define the dynamics $T$ on $X$ by $T=\psi^{-1}\circ \sigma \circ \psi$. 
It is easy to see that $T$ is independent for events of the form $E=\psi^{-1}[w]$.
Since $\{\psi^{-1}[w]:w\in 2^*\}$ 
form an essential family of almost decidable events, $T$ is independent too. 
As $\psi(x)$ is not $\sigma$-typical, 
$x$ is not $T$-typical either.
\end{proof}

\bibliographystyle{plain}
\bibliography{bibliography}

\begin{thebibliography}{10}

\bibitem{Bil68}
Patrick Billingsley.
\newblock {\em Convergence of Probability Measures}.
\newblock John Wiley, New York, 1968.

\bibitem{Dow02}
Rodney~G. Downey and Evan~J. Griffiths.
\newblock Schnorr randomness.
\newblock {\em Electr. Notes Theor. Comput. Sci.}, 66(1), 2002.

\bibitem{Gac05}
Peter G{\'{a}}cs.
\newblock Uniform test of algorithmic randomness over a general space.
\newblock {\em Theoretical Computer Science}, 341:91--137, 2005.

\bibitem{GalHoyRoj1}
Stefano Galatolo, Mathieu Hoyrup, and Crist\'obal Rojas.
\newblock Effective symbolic dynamics, random points, statistical behavior,
  complexity and entropy.
\newblock {\em Submitted}, 2007.

\bibitem{HoyRoj07}
Mathieu Hoyrup and Crist\'obal Rojas.
\newblock Computability of probability measures and {M}artin-{L}\"of randomness
  over metric spaces.
\newblock {\em Information and Computation}, 2009.
\newblock To appear.

\bibitem{Iso05}
Stefano Isola.
\newblock On systems with finite ergodic degree.
\newblock {\em Far east journal of dynamical systems}, 5:1, 2003.

\bibitem{Kol65}
Andrey~N. Kolmogorov.
\newblock Three approaches to the quantitative definition of information.
\newblock {\em Problems in Information Transmission}, 1:1--7, 1965.

\bibitem{MLof66}
Per Martin-L{\"o}f.
\newblock The definition of random sequences.
\newblock {\em Information and Control}, 9(6):602--619, 1966.

\bibitem{Sch71}
Claus-Peter Schnorr.
\newblock {\em Zuf\"alligkeit und {W}ahrscheinlichkeit. Eine algorithmische
  {B}egr\"undung der {W}ahrscheinlichkeitstheorie}, volume 218 of {\em Lecture
  Notes in Mathematics}.
\newblock Springer-Verlag, Berlin-New York, 1971.

\bibitem{Sch72}
Claus-Peter Schnorr.
\newblock The process complexity and effective random tests.
\newblock In {\em STOC}, pages 168--176, 1972.

\bibitem{Tur36}
Alan Turing.
\newblock On computable numbers, with an application to the
  {E}ntscheidungsproblem.
\newblock {\em Proceedings of the London Mathematical Society}, 2, 42:230--265,
  1936.

\bibitem{Via97}
Marcelo Viana.
\newblock Stochastic dynamics of deterministic systems.
\newblock {\em Lecture Notes XXI Braz. Math. Colloq. IMPA Rio de Janeiro},
  1997.

\bibitem{Ville39}
J.~Ville.
\newblock {\em Etude Critique de la Notion de Collectif}.
\newblock Gauthier-Villars, Paris, 1939.

\bibitem{Vyu97}
V'yugin V.V.
\newblock Effective convergence in probability and an ergodic theorem for
  individual random sequences.
\newblock {\em SIAM Theory of Probability and Its Applications}, 42(1):39--50,
  1997.

\bibitem{Wei93}
Klaus Weihrauch.
\newblock Computability on computable metric spaces.
\newblock {\em Theoretical Computer Science}, 113:191--210, 1993.
\newblock Fundamental Study.

\bibitem{Y}
Lai-Sang Young.
\newblock What are {SRB} measures, and which dynamical systems have them?
\newblock {\em J. Stat. Phys.}, 108:733--754, 2002.

\bibitem{LevZvo70}
A.K. Zvonkin and L.A. Levin.
\newblock The complexity of finite objects and the development of the concepts
  of information and randomness by means of the theory of algorithms.
\newblock {\em Russian Mathematics Surveys}, 256:83--124, 1970.

\end{thebibliography}
\end{document}